\documentclass[10pt]{article}         
\usepackage{times, url}
\usepackage{amsmath,amssymb,amsthm,amsfonts}   

\newtheorem{theorem}{Theorem}[section]

\newtheorem{lemma}[theorem]{Lemma}
\newtheorem{remark}[theorem]{Remark}

\title{Some Modular Considerations Regarding Odd Perfect Numbers}
\author{Jose Arnaldo Bebita Dris, Immanuel Tobias San Diego}

\begin{document}

\maketitle

\begin{abstract}
Let $p^k m^2$ be an odd perfect number with special prime $p$.  In this article, we provide an alternative proof for the biconditional that $\sigma(m^2) \equiv 1 \pmod 4$ holds if and only if $p \equiv k \pmod 8$.  We then give an application of this result to the case when $\sigma(m^2)/p^k$ is a square.
\end{abstract}

\section{Introduction}

Let $\sigma(z)$ denote the sum of the divisors of $z \in \mathbb{N}$, the set of positive integers.  Denote the deficiency \cite{Sloane} of $z$ by $D(z)=2z-\sigma(z)$, and the sum of the aliquot divisors \cite{SloaneGuy} of $z$ by $s(z)=\sigma(z)-z$.  Note that we have the identity $D(z) + s(z) = z$.

If $n$ is odd and $\sigma(n)=2n$, then $n$ is said to be an odd perfect number \cite{Wikipedia}.  Euler proved that an odd perfect number, if one exists, must have the form $n = p^k m^2$, where $p$ is the special prime satisfying $p \equiv k \equiv 1 \pmod 4$ and $\gcd(p,m)=1$.

Chen and Luo \cite{ChenLuo} gave a characterization of the forms of odd perfect numbers $n = p^k m^2$ such that $p \equiv k \pmod 8$.  Starni \cite{Starni} proved that there is no odd perfect number decomposable into primes all of the type $\equiv 1 \pmod 4$ if $n = p^k m^2$ and $p \not\equiv k \pmod 8$.  Starni used a congruence from Ewell \cite{Ewell} to prove this result.

Note that, in general, since $m^2$ is a square, we get
$$\sigma(m^2) \equiv 1 \pmod 2.$$

This paper provides an alternative proof for Theorem 3.3, equation 3.1 in Chen and Luo's article titled ``Odd multiperfect numbers" \cite{ChenLuo}:
\begin{theorem}\label{ChenLuo}
Let $n = {{\pi}^{\alpha}}{M^2}$ be an odd $2$-perfect number, with $\pi$ prime, $\gcd(\pi,M)=1$ and $\pi \equiv \alpha \equiv 1 \pmod 4$.  Then
$$\sigma(M^2) \equiv 1 \pmod 4 \iff \pi \equiv \alpha \pmod 8.$$
\end{theorem}

The method presented in this paper may potentially be used to extend the arguments to consider $\sigma(m^2)$ modulo $8$.

\section{Preliminaries}

Starting from the fundamental equality
$$\frac{\sigma(m^2)}{p^k} = \frac{2m^2}{\sigma(p^k)}$$
(which follows from the facts that $\sigma(n)=2n$, $\sigma$ is multiplicative, and $\gcd(p^k,\sigma(p^k))=1$) one can derive
$$\frac{\sigma(m^2)}{p^k} = \frac{2m^2}{\sigma(p^k)} = \gcd(m^2, \sigma(m^2))$$
so that we ultimately have
$$\frac{D(m^2)}{s(p^k)} = \frac{2m^2 - \sigma(m^2)}{\sigma(p^k) - p^k} = \gcd(m^2, \sigma(m^2))$$
and
$$\frac{s(m^2)}{D(p^k)/2} = \frac{\sigma(m^2) - m^2}{p^k - \frac{\sigma(p^k)}{2}} = \gcd(m^2, \sigma(m^2)),$$
whereby we obtain
$$\frac{D(p^k)D(m^2)}{s(p^k)s(m^2)} = 2.$$
Note that we also have the following equation:
$$\frac{2D(m^2)s(m^2)}{D(p^k)s(p^k)} = \bigg(\gcd(m^2, \sigma(m^2))\bigg)^2. \hspace{1.0in} (*)$$ 
Lastly, notice that we can easily get
$$\sigma(p^k) \equiv k + 1 \equiv 2 \pmod 4$$
(since $p \equiv k \equiv 1 \pmod 4$) so that it remains to consider the possible equivalence classes for $\sigma(m^2)$ modulo $4$.  Since $\sigma(m^2)$ is odd, we only need to consider two.

We ask:  Which equivalence class of $\sigma(m^2)$ modulo $4$ makes Equation $(*)$ untenable?

\section{Discussion and Results}

We know that the answer to the question we posed in the previous section must somehow depend on the equivalence class of $p$ and $k$ modulo $8$, but as we only know that $p \equiv k \equiv 1 \pmod 4$, we need to consider the following cases separately and thereby prove the corresponding results:

\begin{remark}\label{MainTheoremRemark}
Suppose that $n = p^k m^2$ is an odd perfect number with special prime $p$. We claim the truth of the following propositions, which we will need to treat separately later:
\begin{enumerate}
\item{If $p \equiv k \equiv 1 \pmod 8$, then $\sigma(m^2) \equiv 3 \pmod 4$ is impossible.}
\item{If $p \equiv 1 \pmod 8$ and $k \equiv 5 \pmod 8$, then $\sigma(m^2) \equiv 1 \pmod 4$ is impossible.}
\item{If $p \equiv 5 \pmod 8$ and $k \equiv 1 \pmod 8$, then $\sigma(m^2) \equiv 1 \pmod 4$ is impossible.}
\item{If $p \equiv k \equiv 5 \pmod 8$, then $\sigma(m^2) \equiv 3 \pmod 4$ is impossible.}
\end{enumerate}
\end{remark}

First, we prove the following lemmas:

\begin{lemma}\label{Lemma1}
Suppose that $n = p^k m^2$ is an odd perfect number with special prime $p$.
\begin{enumerate}
\item{If $p \equiv 1 \pmod 8$, then $\sigma(p^k) \equiv k + 1 \pmod 8$.}
\item{If $p \equiv 5 \pmod 8$ and $k \equiv 1 \pmod 8$, then $\sigma(p^k) \equiv 6 \pmod 8$.}
\item{If $p \equiv 5 \pmod 8$ and $k \equiv 5 \pmod 8$, then $\sigma(p^k) \equiv 2 \pmod 8$.}
\end{enumerate}
\end{lemma}

\begin{proof}
Let $n = p^k m^2$ be an odd perfect number with special prime $p$.  It follows that $p \equiv 1 \pmod 4$.

We consider two cases:

{\bf Case 1}: $p \equiv 1 \pmod 8$
We obtain
$$\sigma(p^k) = \sum_{i=0}^{k}{p^i} \equiv 1 + \sum_{i=1}^{k}{p^i} \equiv 1 + \sum_{i=1}^{k}{1^i} \equiv k + 1 \pmod 8,$$
as desired.

{\bf Case 2}: $p \equiv 5 \pmod 8$
We get
$$\sigma(p^k) = \sum_{i=0}^{k}{p^i} \equiv \sum_{i=0}^{k}{5^i} \equiv 
\begin{cases}
6 \pmod 8, \text{ if } k \equiv 1 \pmod 8 \\
2 \pmod 8, \text{ if } k \equiv 5 \pmod 8 \\
\end{cases}
$$
\end{proof}

\begin{lemma}\label{Lemma2}
Suppose that $n = p^k m^2$ is an odd perfect number with special prime $p$.
\begin{enumerate}
\item{If $p \equiv 1 \pmod 8$ and $k \equiv 1 \pmod 8$, then $D(p^k) \equiv 0 \pmod 8$.}
\item{If $p \equiv 1 \pmod 8$ and $k \equiv 5 \pmod 8$, then $D(p^k) \equiv 4 \pmod 8$.}
\item{If $p \equiv 5 \pmod 8$ and $k \equiv 1 \pmod 8$, then $D(p^k) \equiv 4 \pmod 8$.}
\item{If $p \equiv 5 \pmod 8$ and $k \equiv 5 \pmod 8$, then $D(p^k) \equiv 0 \pmod 8$.}
\end{enumerate}
\end{lemma}

\begin{proof}
The proof is trivial and follows directly from Lemma \ref{Lemma1}, using the formula $D(p^k)=2p^k - \sigma(p^k)$.
\end{proof}

\begin{lemma}\label{Lemma3}
Suppose that $n = p^k m^2$ is an odd perfect number with special prime $p$.
\begin{enumerate}
\item{If $p \equiv 1 \pmod 8$ and $k \equiv 1 \pmod 8$, then $s(p^k) \equiv 1 \pmod 8$.}
\item{If $p \equiv 1 \pmod 8$ and $k \equiv 5 \pmod 8$, then $s(p^k) \equiv 5 \pmod 8$.}
\item{If $p \equiv 5 \pmod 8$ and $k \equiv 1 \pmod 8$, then $s(p^k) \equiv 1 \pmod 8$.}
\item{If $p \equiv 5 \pmod 8$ and $k \equiv 5 \pmod 8$, then $s(p^k) \equiv 5 \pmod 8$.}
\end{enumerate}
\end{lemma}

\begin{proof}
The proof is trivial and follows directly from Lemma \ref{Lemma2}, using the formula $s(p^k) = p^k - D(p^k)$.
\end{proof}

\begin{lemma}\label{Lemma4}
Suppose that $n = p^k m^2$ is an odd perfect number with special prime $p$.
\begin{enumerate}
\item{If $\sigma(m^2) \equiv 1 \pmod 4$, then $D(m^2) \equiv 1 \pmod 4$.}
\item{If $\sigma(m^2) \equiv 3 \pmod 4$, then $D(m^2) \equiv 3 \pmod 4$.}
\end{enumerate}
\end{lemma}

\begin{proof}
The proof is trivial and follows directly from the fact that $m^2 \equiv 1 \pmod 4$ (since $m$ is odd), using the underlying assumptions and the formula $D(m^2)=2m^2 - \sigma(m^2)$.
\end{proof}

\begin{lemma}\label{Lemma5}
Suppose that $n = p^k m^2$ is an odd perfect number with special prime $p$.
\begin{enumerate}
\item{If $\sigma(m^2) \equiv 1 \pmod 4$, then $s(m^2) \equiv 0 \pmod 4$.}
\item{If $\sigma(m^2) \equiv 3 \pmod 4$, then $s(m^2) \equiv 2 \pmod 4$.}
\end{enumerate}
\end{lemma}

\begin{proof}
The proof is trivial and follows directly from Lemma \ref{Lemma4}, using the formula $s(m^2) = m^2 - D(m^2)$.
\end{proof}

We are now ready to prove our main result.

\begin{theorem}\label{MainTheorem}
Suppose that $n = p^k m^2$ is an odd perfect number with special prime $p$.
\begin{enumerate}
\item{If $p \equiv k \equiv 1 \pmod 8$, then $\sigma(m^2) \equiv 3 \pmod 4$ is impossible.}
\item{If $p \equiv 1 \pmod 8$ and $k \equiv 5 \pmod 8$, then $\sigma(m^2) \equiv 1 \pmod 4$ is impossible.}
\item{If $p \equiv 5 \pmod 8$ and $k \equiv 1 \pmod 8$, then $\sigma(m^2) \equiv 1 \pmod 4$ is impossible.}
\item{If $p \equiv k \equiv 5 \pmod 8$, then $\sigma(m^2) \equiv 3 \pmod 4$ is impossible.}
\end{enumerate}
\end{theorem}

\begin{proof}
Let $n = p^k m^2$ be an odd perfect number with special prime $p$.

Notice that the right-hand side of Equation $(*)$
$$\frac{2D(m^2)s(m^2)}{D(p^k)s(p^k)} = \bigg(\gcd(m^2, \sigma(m^2))\bigg)^2. \hspace{1.0in} (*)$$ 
is odd.  (Furthermore, it is congruent to $1$ modulo $8$.)

First, suppose that $p \equiv k \equiv 1 \pmod 8$, and assume to the contrary that $\sigma(m^2) \equiv 3 \pmod 4$ holds.  By Lemma \ref{Lemma2}, $D(p^k) \equiv 0 \pmod 8$.  By Lemma \ref{Lemma4}, $D(m^2) \equiv 3 \pmod 4$.  By Lemma \ref{Lemma3}, $s(p^k) \equiv 1 \pmod 8$.  By Lemma \ref{Lemma5}, $s(m^2) \equiv 2 \pmod 4$.  Thus, from Equation $(*)$ we obtain (symbolically)
$$2(4a_1 + 3)(4b_1 + 2) = (8x_1 + 1)(8c_1)(8d_1 + 1)$$
which does not have any integer solutions.

Next, suppose that $p \equiv 1 \pmod 8$ and $k \equiv 5 \pmod 8$, and assume to the contrary that $\sigma(m^2) \equiv 1 \pmod 4$ holds.  By Lemma \ref{Lemma2}, $D(p^k) \equiv 4 \pmod 8$.  By Lemma \ref{Lemma4}, $D(m^2) \equiv 1 \pmod 4$.  By Lemma \ref{Lemma3}, $s(p^k) \equiv 5 \pmod 8$.  By Lemma \ref{Lemma5}, $s(m^2) \equiv 0 \pmod 4$.  Thus, from Equation $(*)$ we obtain (symbolically)
$$2(4a_2 + 1)(4b_2) = (8x_2 + 1)(8c_2 + 4)(8d_2 + 5)$$
which does not have any integer solutions.

Now, suppose that $p \equiv 5 \pmod 8$ and $k \equiv 1 \pmod 8$, and assume to the contrary that $\sigma(m^2) \equiv 1 \pmod 4$ holds.  By Lemma \ref{Lemma2}, $D(p^k) \equiv 4 \pmod 8$.  By Lemma \ref{Lemma4}, $D(m^2) \equiv 1 \pmod 4$.  By Lemma \ref{Lemma3}, $s(p^k) \equiv 1 \pmod 8$.  By Lemma \ref{Lemma5}, $s(m^2) \equiv 0 \pmod 4$.  Thus, from Equation $(*)$ we obtain (symbolically)
$$2(4a_3 + 1)(4b_3) = (8x_3 + 1)(8c_3 + 4)(8d_3 + 1)$$
which does not have any integer solutions.

Finally, suppose that $p \equiv k \equiv 5 \pmod 8$, and assume to the contrary that $\sigma(m^2) \equiv 3 \pmod 4$ holds.  By Lemma \ref{Lemma2}, $D(p^k) \equiv 0 \pmod 8$.  By Lemma \ref{Lemma4}, $D(m^2) \equiv 3 \pmod 4$.  By Lemma \ref{Lemma3}, $s(p^k) \equiv 5 \pmod 8$.  By Lemma \ref{Lemma5}, $s(m^2) \equiv 2 \pmod 4$.  Thus, from Equation $(*)$ we obtain (symbolically)
$$2(4a_4 + 3)(4b_4 + 2) = (8x_4 + 1)(8c_4)(8d_4 + 5)$$
which does not have any integer solutions.

This concludes the proof.
\end{proof}

\begin{remark}\label{MainThmRemark}
To summarize, Theorem \ref{MainTheorem} just states that if $n = p^k m^2$ is an odd perfect number with special prime $p$, then $\sigma(m^2) \equiv 1 \pmod 4$ holds if and only if $p \equiv k \pmod 8$.  Our argument provides an alternative proof for Theorem 3.3, equation 3.1 in \cite{ChenLuo} (as reproduced above in Theorem \ref{ChenLuo}). 
\end{remark}

\section{An Application}

Let $n = p^k m^2$ be an odd perfect number with special prime $p$, and let $\sigma(m^2)/p^k$ be a square.  Since $\sigma(m^2)/p^k$ is odd, it follows that $\sigma(m^2)/p^k \equiv 1 \pmod 4$.  But it is known that $p \equiv k \equiv 1 \pmod 4$.  In particular, we know that $p^k \equiv 1 \pmod 4$.  This implies that $\sigma(m^2) \equiv 1 \pmod 4$, if $\sigma(m^2)/p^k$ is a square.  By Theorem \ref{MainTheorem}, we know that $p \equiv k \pmod 8$.

Moreover, Broughan, Delbourgo, and Zhou proved in \cite{BroughanDelbourgoZhou} (Lemma $8$, page $7$) that if $\sigma(m^2)/p^k$ is a square, then $k=1$ holds.

Thus, under the assumption that $\sigma(m^2)/p^k$ is a square, we have
$$p \equiv k = 1 \pmod 8.$$
This implies that the lowest possible value for the special prime $p$ is $17$.

We state this result as our next theorem.

\begin{theorem}\label{Application}
Suppose that $n = p^k m^2$ is an odd perfect number with special prime $p$.  If $\sigma(m^2)/p^k$ is a square, then $p \geq 17$.
\end{theorem}

\begin{remark}\label{OchemRemark}
Let $n = p^k m^2$ be an odd perfect number with special prime $p$.

Note that if
$$\frac{\sigma(m^2)}{p^k}=\frac{m^2}{\sigma(p^k)/2}$$
is a square, then $k=1$ and $\sigma(p^k)/2 = (p+1)/2$ is also a square.

The possible values for the special prime satisfying $p < 100$ and $p \equiv 1 \pmod 8$ are $17$, $41$, $73$, $89$, and $97$.

For each of these values:
$$\frac{p_1 + 1}{2} = \frac{17 + 1}{2} = 9 = 3^2.$$
$$\frac{p_2 + 1}{2} = \frac{41 + 1}{2} = 21 \text{ which is not a square.}$$
$$\frac{p_3 + 1}{2} = \frac{73 + 1}{2} = 37 \text{ which is not a square.}$$
$$\frac{p_4 + 1}{2} = \frac{89 + 1}{2} = 45 \text{ which is not a square.}$$
$$\frac{p_5 + 1}{2} = \frac{97 + 1}{2} = 49 = 7^2.$$

A quick way to rule out $41$, $73$ and $89$, as remarked by Ochem \cite{Ochem} over at Mathematics StackExchange, is as follows:  ``If $(p+1)/2$ is an odd square, then $(p+1)/2 \equiv 1 \pmod 8$, so that $p \equiv 1 \pmod {16}$. This rules out $41$, $73$, and $89$."
\end{remark}

\section{Conclusion}

Additional tools are required if we are to push the analysis from $\sigma(m^2)$ modulo $4$ to consider $\sigma(m^2)$ modulo $8$.  The authors have tried to check Equation $(*)$ by considering $m^2 \equiv 1 \pmod 8$, and the various corresponding cases for $\sigma(m^2)$ modulo $8$ (which are determined by Theorem \ref{MainTheorem}), but so far all their attempts have not resulted in any contradictions.

\section*{Acknowledgements} 

The authors are indebted to the anonymous referees whose valuable feedback improved the overall presentation and style of this manuscript.

\makeatletter
\renewcommand{\@biblabel}[1]{[#1]\hfill}
\makeatother

\end{document}